\documentclass[
               letterpaper,%
               twoside,%
               10pt%
               ]{article}

\usepackage[hmargin={1.0in,1.0in},
            vmargin={1.0in,1.0in},%
            headheight=14pt,
            ]{geometry}

\usepackage{amsmath}
\numberwithin{equation}{section}                        
\usepackage{amsthm}
\usepackage{amssymb}                                    
\usepackage{bbm}                                        
\makeatletter
\let\@upn\@iden
\makeatother
\newtheoremstyle{ToddTheorem}                           
  {3pt}
  {3pt}
  {\slshape}
  {}
  {\bfseries\upshape}
  {.}
  {.5em}
  {}

\newtheoremstyle{ToddDefinition}                        
  {3pt}
  {3pt}
  {\upshape}
  {}
  {\bfseries\slshape}
  {.}
  {.5em}
  {}

\makeatletter
\newcommand*{\c@theoremassumption}{\c@theorem}
\newcommand*{\p@theoremassumption}{\p@theorem}

\makeatother

\makeatletter
\newcommand*{\c@theoremconjecture}{\c@theorem}
\newcommand*{\p@theoremconjecture}{\p@theorem}

\makeatother

\makeatletter
\newcommand*{\c@theoremcorollary}{\c@theorem}
\newcommand*{\p@theoremcorollary}{\p@theorem}

\makeatother

\makeatletter
\newcommand*{\c@theoremdefinition}{\c@theorem}
\newcommand*{\p@theoremdefinition}{\p@theorem}

\makeatother

\makeatletter
\newcommand*{\c@theoremexample}{\c@theorem}
\newcommand*{\p@theoremexample}{\p@theorem}

\makeatother

\makeatletter
\newcommand*{\c@theoremfigure}{\c@theorem}
\newcommand*{\p@theoremfigure}{\p@theorem}

\makeatother

\makeatletter
\newcommand*{\c@theoremhypothesis}{\c@theorem}
\newcommand*{\p@theoremhypothesis}{\p@theorem}

\makeatother

\makeatletter
\newcommand*{\c@theoremlemma}{\c@theorem}
\newcommand*{\p@theoremlemma}{\p@theorem}

\makeatother

\makeatletter
\newcommand*{\c@theoremproposition}{\c@theorem}
\newcommand*{\p@theoremproposition}{\p@theorem}

\makeatother

\makeatletter
\newcommand*{\c@theoremremark}{\c@theorem}
\newcommand*{\p@theoremremark}{\p@theorem}

\makeatother

\makeatletter
\newcommand*{\c@theoremnotation}{\c@theorem}
\newcommand*{\p@theoremnotation}{\p@theorem}

\makeatother

\newtheoremstyle{SpringerTheorem}                           
  {3pt}
  {3pt}
  {\itshape}
  {}
  {\bfseries}
  {.}
  {.5em}
  {}
\newtheoremstyle{SpringerDefinition}                        
  {3pt}
  {3pt}
  {\rmfamily}
  {}
  {\bfseries}
  {.}
  {.5em}
  {}
\newtheoremstyle{SpringerExample}                        
  {3pt}
  {3pt}
  {\rmfamily}
  {}
  {\itshape}
  {.}
  {.5em}
  {}

\theoremstyle{SpringerTheorem} %
\newtheorem{theorem}{Theorem}[section]
\newtheorem{corollary}[theoremcorollary]{Corollary}
\newtheorem{lemma}[theoremlemma]{Lemma}
\newtheorem{proposition}[theoremproposition]{Proposition}

\theoremstyle{SpringerDefinition} %

\theoremstyle{SpringerExample} %

\newtheorem{remark}[theoremremark]{Remark}

\newcommand{\TheAuthor}{}
\newcommand{\Author}[1]%
        {\renewcommand{\TheAuthor}{#1}}                 
\newcommand{\TheRunningTitle}{}
\newcommand{\RunningTitle}[1]%
        {\renewcommand{\TheRunningTitle}{#1}}           

\usepackage{fancyhdr}                                   
\renewcommand{\footrulewidth}{0.4pt}

\usepackage{titling}
\pretitle{\begin{center}\LARGE\sffamily\slshape}        
\posttitle{\par\end{center}\vskip 0.5em}
\predate{\begin{center}\small}                          
\postdate{\par\end{center}\vskip 0.5em}
\thanksmarkseries{fnsymbol}                             

\usepackage[runin]{abstract}
\abslabeldelim{.\,}

\usepackage[titles]{tocloft}                            

\usepackage{titlesec}
\titleformat{\section}%
   {\Large\scshape\filcenter}{\thesection.}{1em}{}
\titleformat{\subsection}%
   {\large\sffamily\slshape\filcenter}{\thesubsection.}{1em}{}
\titleformat{\subsubsection}%
   {\normalsize\sffamily\filcenter}{\thesubsubsection.}{1em}{}
\titlespacing*{\section}%
   {0pt}{4.5ex plus 1ex minus .2ex}{3.3ex plus .2ex}
\titlespacing*{\subsection}%
   {0pt}{4.25ex plus 1ex minus .2ex}{3.05ex plus .2ex}
\titlespacing*{\subsubsection}%
   {0pt}{4.25ex plus 1ex minus .2ex}{3.05ex plus .2ex}

\usepackage[comma,numbers,square,sort&compress]{natbib} 

\usepackage{bm}                                         

\usepackage[
            ]{graphicx}                                 
\usepackage[
            small,%
            bf]{caption}                     
\setcaptionmargin{0.5in}

\newenvironment{acknowledgment}%
  {\begin{trivlist}\item[]\textbf{Acknowledgments.}}{\end{trivlist}}

\makeatletter
\providecommand*{\toclevel@assumption}{0}
\providecommand*{\toclevel@conjecture}{0}
\providecommand*{\toclevel@corollary}{0}
\providecommand*{\toclevel@definition}{0}
\providecommand*{\toclevel@example}{0}
\providecommand*{\toclevel@figure}{0}
\providecommand*{\toclevel@hypothesis}{0}
\providecommand*{\toclevel@lemma}{0}
\providecommand*{\toclevel@proof}{0}
\providecommand*{\toclevel@proposition}{0}
\providecommand*{\toclevel@remark}{0}
\providecommand*{\toclevel@theorem}{0}
\makeatother

\newcommand{\eref}[1]{\hyperref[{#1}]{(\ref*{#1})}}

\newcommand{\R}{\mathbb{R}}

\def\dim{\mathop\mathrm{dim}\nolimits}

\def\ess{\mathop\mathrm{ess}\nolimits}

\def\gker{\mathop\mathrm{gker}\nolimits}
\def\Ham{\mathop\mathrm{Ham}\nolimits}
\def\id{\mathcal{I}}
\def\ker{\mathop\mathrm{ker}\nolimits}

\def\sgn{\mathop\mathrm{sign}\nolimits}
\def\Span{\mathop\mathrm{span}\nolimits}

\def\Re{\mathop\mathrm{Re}\nolimits}
\def\Im{\mathop\mathrm{Im}\nolimits}

\def\coloneqq{\mathrel{\mathop:}=}

\newcommand{\rmc}{\mathrm{c}}
\newcommand{\rmd}{\mathrm{d}}
\newcommand{\rme}{\mathrm{e}}

\newcommand{\rmi}{\mathrm{i}}

\newcommand{\rmn}{\mathrm{n}}

\newcommand{\rmp}{\mathrm{p}}
\newcommand{\rmq}{\mathrm{q}}
\newcommand{\rmr}{\mathrm{r}}

\newcommand{\rmv}{\mathrm{v}}

\newcommand{\calH}{\mathcal{H}}
\newcommand{\calI}{\mathcal{I}}

\newcommand{\calJ}{\mathcal{J}}
\newcommand{\calK}{\mathcal{K}}

\newcommand{\calL}{\mathcal{L}}
\newcommand{\calM}{\mathcal{M}}

\newcommand{\calS}{\mathcal{S}}

\newcommand{\vD}{\bm{\mathit{D}}}

\newcommand{\ve}{\bm{\mathit{e}}}

\newcommand{\vp}{\bm{\mathit{p}}}

\newcommand{\cm}{\mathscr{C}}

\newcommand{\f}[2]{\frac{#1}{#2}}
\newcommand{\dpr}[2]{\langle #1,#2 \rangle}

\newcommand{\al}{\alpha}

\newcommand{\de}{\delta}

\renewcommand{\ve}{\varepsilon}

\newcommand{\ka}{\kappa}
\newcommand{\la}{\lambda}

\newcommand{\si}{\sigma}

\renewcommand{\vp}{\varphi}

\newcommand{\tL}{\calL^\diamondsuit}

\newcommand{\cs}{\mathcal S}

\newcommand{\cl}{\mathcal L}
\renewcommand{\cm}{\mathcal M}

\newcommand{\p}{\partial}

\newcommand{\beq}{\begin{equation}}
\newcommand{\eeq}{\end{equation}}
\newcommand{\beqna}{\begin{eqnarray*}}
\newcommand{\eeqna}{\end{eqnarray*}}
\newcommand{\beqn}{\begin{equation*}}
\newcommand{\eeqn}{\end{equation*}}
\newcommand{\bp}{\begin{proof}}
\newcommand{\ep}{\end{proof}}
\newcommand{\bprop}{\begin{proposition}}
\newcommand{\eprop}{\end{proposition}}
\newcommand{\bt}{\begin{theorem}}
\newcommand{\et}{\end{theorem}}
\newcommand{\bex}{\begin{Example}}
\newcommand{\eex}{\end{Example}}
\newcommand{\bc}{\begin{corollary}}
\newcommand{\ec}{\end{corollary}}
\newcommand{\bcl}{\begin{claim}}
\newcommand{\ecl}{\end{claim}}
\newcommand{\bl}{\begin{lemma}}
\newcommand{\el}{\end{lemma}}

\Author{T. Kapitula and A. Stefanov}
\RunningTitle{Index theory for KdV-like problems}

\usepackage[letterpaper,%
            bookmarks=false,%
            bookmarksopen=false,%
            bookmarksnumbered=false,%
            hypertexnames=false,%
            colorlinks=true,%
            linkcolor=blue,%
            citecolor=blue,%
            urlcolor=blue]{hyperref}                    

\usepackage{kpfonts}

\begin{document}

\title{A Hamiltonian-Krein (instability) index theory for KdV-like eigenvalue problems}
\author{%
        Todd Kapitula %
        \thanks{E-mail: \href{mailto:tmk5@calvin.edu}{tmk5@calvin.edu}} \\
        Department of Mathematics and Statistics \\
        Calvin College \\
        Grand Rapids, MI 49546 %
\and
        Atanas Stefanov %
        \thanks{E-mail: \href{mailto:stefanov@math.ku.edu}{stefanov@math.ku.edu}} \\
        Department of Mathematics \\
        University of Kansas \\
        Lawrence, KS 66045-7594 %
}


\begin{titlingpage}
\usethanksrule
\setcounter{page}{0}                                    
\maketitle
\begin{abstract}
The Hamiltonian-Krein (instability) index is concerned with determining the
number of eigenvalues with positive real part for the Hamiltonian
eigenvalue problem $\calJ\calL u=\lambda u$, where $\calJ$ is
skew-symmetric and $\calL$ is self-adjoint. If $\calJ$ has a bounded
inverse the index is well-established, and it is given by the number of
negative eigenvalues of the operator $\calL$ constrained to act on some
finite-codimensional subspace. There is an important class of problems -
namely, those of KdV-type - for which $\calJ$ does not have a bounded
inverse. In this paper we overcome this difficulty and derive the index for
eigenvalue problems of KdV-type. We use the index to discuss the spectral
stability of homoclinic traveling waves for KdV-like problems and BBM-type
problems.
\end{abstract}

\cancelthanksrule
\renewcommand{\footrulewidth}{0.0pt}                    

\pdfbookmark[1]{\contentsname}{toc}                     
\tableofcontents                                        
\end{titlingpage}

\section{Introduction}

We consider the spectral problem of the form
\begin{equation}\label{5}
\partial_x\calL u= \lambda u,
\end{equation}
where $\calL$ is a self-adjoint linear differential Fredholm operator with
zero index and with domain $D(\calL)=H^s(\R)$ for some $s\geq 0$. Eigenvalue
problems of this type readily arise, e.g., when considering the stability of
waves to KdV-like problems. We will furthermore assume that
$\cl=\calL_0+ \calK$, where $\calK$ is relatively compact perturbation of
$\calL_0$, and $\calL_0$ is a ``constant coefficient'' strongly elliptic
operator\footnote{that is, $q(\xi)\geq \ka^2$ for some $\ka>0$} given by
$\widehat{\calL_0 f}(\xi)=q(\xi)\hat{f}(\xi)$. It will be assumed that for
the operator $\calL$,
\begin{enumerate}
\item there are $\rmn(\calL)<+\infty$ negative eigenvalues (counting
    multiplicity), and each of the corresponding eigenvectors
    $\{f_j\}_{j=1}^{\rmn(\calL)}$ belong to $H^{1/2}(\R)$.
\item
    $\sigma_{\ess}(\calL)=\sigma_{\ess}(\calL_0)=\mathrm{Range}(q)\subset[\kappa^2,+\infty), \ka>0$
\item $\dim[\ker(\calL)]=1$ with $\ker(\calL)=\Span\{\psi_0\}$, and
$\psi_0$ is real-valued and  $\psi_0\in H^{\infty}(\R)\cap \dot{H}^{-1}(\R)$.
\end{enumerate}
For the precise definitions of the various Sobolev spaces, consult
\autoref{sec:2}.

The goal of this paper is to compute an instability index, hereafter known as
the Hamiltonian-Krein index, for the eigenvalue problem \eref{5}. In the
derivation of instability indices for eigenvalue problems of the form
\[
\calJ\calL u=\lambda u,
\]
where $\calJ$ is skew-symmetric and $\calL$ is symmetric, it was crucial in
previous works that $\calJ$ have a bounded inverse on (at minimum) a finite
co-dimensional space (e.g., see
\citep{kapitula:cev04,kapitula:ace05,haragus:ots08,deconinck:ots10}). Define
the standard inner-product on $L^2(\R)$ by
\[
\langle f,g\rangle=\int_{-\infty}^{+\infty}f(x)\overline{g}(x)\,\rmd x.
\]
It is clear that the operator $\partial_x$ is skew-symmetric on $L^2(\R)$;
however, it does not have a bounded inverse. This is a reflection of the fact
that $\sigma(\partial_x)=\sigma_{\ess}(\partial_x)=\rmi\R$. The aim of this
paper is to overcome this obstacle. Briefly, this will be accomplished by
reducing the eigenvalue problem \eref{5} to an equivalent problem for which
the operator $\calJ$ does have a bounded inverse. However, by doing so it
will be the case that for the new operator $\calL$:
\begin{enumerate}
\item the essential spectrum will (generically) be $[0,\infty+)$, which
    violates the assumption present in the original computation of the
    Hamiltonian-Krein index that the essential spectrum be bounded away
    from the origin
\item the negative index of the new $\calL$, which is needed in the
    evaluation of the index, is not obvious.
\end{enumerate}
Both of these obstacles must be overcome before coming to the final
conclusion of \autoref{t:index}.

The paper is organized in the following manner. In \autoref{sec:2} we discuss
some preliminary ideas which will be needed in the analysis. The results
presented therein are not new, and are included solely to help make the paper
more accessible. In \autoref{sec:3} the equivalent eigenvalue problem is
derived, and properties of the new operator $\calL$ are given.
\autoref{s:hki} contains the main result of the paper. In \autoref{Sec:4.1} we
give a couple of applications of the theoretical result, and compare the
results here with what is already known in the literature.

\begin{acknowledgment}
TK gratefully acknowledges the support of the Jack and Lois Kuipers Applied
Mathematics Endowment, a Calvin Research Fellowship, and the National Science
Foundation under grant DMS-1108783. AS research is supported in part by
NSF-DMS 0908802.
\end{acknowledgment}

\section{Preliminaries}\label{sec:2}

Define the Fourier transform and its inverse via the formulas
\[
\hat{f}(\xi)=\int_{\R} f(x)\rme^{-  2\pi\rmi x \xi}\,\rmd x,\quad
f(x) =  \int_{\R} \hat{f}(\xi)\rme^{2\pi\rmi x \xi}\,\rmd\xi,
\]
which are valid for functions in the Schwartz class $\cs$. Introduce
fractional order differential operators via the Fourier transform, i.e. for $s\geq 0$,
\[
\widehat{|\p_x|^s f}(\xi):=(2\pi)^{s} |\xi|^s \hat{f}(\xi).
\]
The norm of the Sobolev space $H^s(\R),\,s\geq 0$, is given by
\[
\|f\|_{H^s}\coloneqq\left( \int_{\R} |\hat{f}(\xi)|^2 (1+\xi^2)^{s}\,\rmd\xi\right)^{1/2}.
\]
The space of infinitely smooth functions (with $L^2(\R)$ decay of all
derivatives),  $H^\infty:=\cap_{s=1}^\infty H^s$ is not a Banach space, but
it has a well-understood  Frechet space structure.

We also need to consider operators in the form $|\p_x|^{-\al}$ for some
$\al>0$. Regarding Sobolev spaces of negative order, we introduce the norm
\[
\|f\|_{\dot{H}^{-\al}}\coloneqq\left( \int_{\R} \f{|\hat{f}(\xi)|^2}{|\xi|^{2\al}}\,\rmd\xi\right)^{1/2},
\]
and say that a Schwartz function $f$ belongs to $\dot{H}^{-\al}(\R)$ if $\|f\|_{\dot{H}^{-\al}}$ is finite. The Banach space $\dot{H}^{-\al}(\R)$ is obtained as the completion of the  Schwartz class $\cs$ in this norm.
Note that the Sobolev spaces of negative order will in general contain
distributions\footnote{In fact, one may define $\dot{H}^{-\al}$ as the dual
space to $\dot{H}^{\al}$, with the obvious definitions. In doing that, one needs
to be careful since the ``norm'' $\|\cdot\|_{\dot{H}^{\al}}$  assigns zero value to
the constant functions and thus, those functions need to be mod-ed out.}.   Further note that
\begin{enumerate}
\item $|\p_x|^{-\al}:\dot{H}^{-\al}(\R)\mapsto L^2(\R)$ is an isometry
\item $|\p_x|^{\al}: L^2(\R)\mapsto \dot{H}^{-\al}(\R)$ is an isometry.
\end{enumerate}
Some of these operators have a nice representation as fractional integrals.
For example, (again for Schwartz functions)
\[
|\p_x|^{-1/2} f (x)= \f{1}{\sqrt{2\pi}} \int_{\R} \f{f(y)}{|x-y|^{1/2}}\,\rmd y;
\]
in particular, $|\p_x|^{-1/2} f$ is real-valued if $f$ is. Note that unless
$f$ has extra cancellation properties\footnote{At a minimum $\int f =0$, but
actually more, like $f\in {\mathcal H^1}(\R)$ - the Hardy space on the line},
then for large values of $x$ one has $|\p_x|^{-1/2} f\sim |x|^{-1/2}$ and
hence $|\p_x|^{-1/2} f\notin L^2(\R)$.

Finally, note that for $f\in \dot{H}^{-1}(\R)$ we may define the operator
$\p_x^{-1}$ via
$$
\widehat{\p_x^{-1} f}(\xi)=-\f{1}{2\pi i \xi} \hat{f}(\xi).
$$
The operator $\p_x^{-1}$ is skew-symmetric, as may be seen by the
Plancherel's theorem. In particular, for every real-valued $f\in
\dot{H}^{-1}(\R)\cap L^2(\R)$ we have that $\dpr{\p_x^{-1} f}{f}=0$. One may
identify $\dot{H}^{-1}(\R)$ as the space of distributional derivatives
$\p_x(L^2(\R))\subset \cs'$. More precisely,
$$
\dot{H}^{-1}(\R)=\p_x(L^2(\R))=\{h:\ h=\p_x f\in \cs',\, f\in L^2(\R)\},\quad
\|h\|_{\dot{H}^{-1}}d\coloneqq\|f\|_{L^2}.
$$

\subsection{Littlewood-Paley operators}

Let $\zeta\in C^\infty_0(\R)$ be a positive and even cut-off function which
satisfies
\[
\zeta(z)=\begin{cases}
 1,\quad & |z|<1 \\
 0,\quad & |z|>2.
\end{cases}
\]
For $a>0$ define the Littlewood-Paley operator $P_{<a}$ via
$$
\widehat{P_{<a} f}(\xi)=\zeta(\xi/a) \hat{f}(\xi).
$$
Naturally, we take $ P_{\geq a}=\id - P_{<a}$, where $\id$ is the identity
operator. The related operators $P_{\sim a}$ are defined via $P_{\sim
a}:=P_{<a}-P_{<\f{a}{2}}$. Alternatively, let $\vp(z):=\zeta(z)-\zeta(2 z)$
and let $ \widehat{P_{\sim a} f}(\xi)=\vp(\xi/a) \hat{f}(\xi). $ Note that by
the Hardy-Littlewood-Sobolev inequality, we have for all $1\leq p\leq \infty$
$$
\|P_{<a} f\|_{L^p}+ \|P_{\geq a} f\|_{L^p}+\\|P_{\sim a} f\|_{L^p}\leq C (1+ \|\hat{\zeta}\|_{L^1}) \|f\|_{L^p}.
$$
We will often denote\footnote{By slight abuse of notations, we will always
use $\widehat{f_{\sim a}}(\xi)\coloneqq\Phi(\xi/a)\hat{f}(\xi)$ and $\Phi$ is
supported around $1$ smooth function}
\[
f_{<a}\coloneqq P_{<a} f,\quad f_{\geq a}\coloneqq P_{\geq a} f,\quad f_{\sim a}=P_{\sim a} f.
\]
Note that the operators $P_{\sim a}$ provide a useful partition of unity.
Indeed, note that $\sum_{k=-\infty}^\infty \vp(2^{-k} \xi)=1$ for $\xi\neq
0$, and as a consequence
$$
\id = \sum_{k=-\infty}^\infty P_{\sim 2^k} = P_{<1}+\sum_{k=1}^\infty P_{\sim 2^k}.
$$
A version of the Sobolev embedding estimates (also known as Bernstein inequalities)  is given by
\begin{equation}
\label{sobol}
\|P_{\sim 2^k} f\|_{L^q}\leq C 2^{k(\f{1}{p}-\f{1}{q})} \|f\|_{L^p}.
\end{equation}
for all $1\leq p< q\leq \infty$.

We have the following lemma:

\begin{lemma}\label{le:dense}
The subspace $\{|\p_x|^{1/2}g: g \in H^{1/2}(\R)\}$ is dense in $L^2(\R)$.
\end{lemma}

\begin{proof}
Let $\ve>0$ and $f\in L^2(\R)$ be given function. Then there exists $\de>0$
so that
$$
\int_{-\de}^\de |\hat{f}(\xi)|^2  d\xi\leq \ve^2.
$$
Define
$$
\hat{g}(\xi)\coloneqq
\begin{cases}{cc}
0,\quad & \ |\xi|\leq \de \\
\hat{f}(\xi)/\sqrt{2\pi |\xi|},\quad   & |\xi|>\de.
\end{cases}
$$
It follows that $g\in H^{1/2}(\R)$ - in fact, $\||\p_x|^{1/2}
g\|_{L^2}^2=\int_{|\xi|>\de} |\hat{f}(\xi)|^2\,\rmd\xi\leq \|f\|_{L^2}^2$ -
while $\|g\|_{L^2}^2\leq\|f\|_{L^2}^2/(2\pi\de)$. In addition, by
Plancherel's
$$
\|f-|\p_x|^{1/2}g\|_{L^2}^2=\int_{-\de}^\de |\hat{f}(\xi)|^2\,\rmd\xi\leq \ve^2.\qedhere
$$
\end{proof}

\section{The equivalent eigenvalue problem}\label{sec:3}

\subsection{The reformulation}

We proceed with the reformulation of the eigenvalue problem \eref{5}. We
first note that for nonzero eigenvalues it will necessarily be the case that
$u\in \dot{H}^{-1}(\R)$ if $u\in D(\calL)=H^s(\R)$. Indeed, from \eref{5}
\[
\|u\|_{\dot{H}^{-1}}=\f{1}{|\la|}\|\calL u\|_{L^2}.
\]
This observation motivates the following change of variables. Set
\[
u=|\p_x|^{1/2} v\quad\Leftrightarrow\quad v=|\p_x|^{-1/2} u.
\]
Note that $v\in \dot{H}^{-1/2}(\R)\cap H^{s+1/2}(\R)$, since
$|\p_x|^{-1/2}:\dot{H}^{-1}(\R)\cap H^s(\R) \mapsto\dot{H}^{-1/2}(\R)\cap H^{s+1/2}(\R)$ is a bounded map. The
eigenvalue problem for $v$ becomes
\[
\p_x |\p_x|^{-1/2} \cl |\p_x|^{1/2} v= \la v,
\]
which can be massaged to
\[
\p_x |\p_x|^{-1}\cdot|\p_x|^{1/2}\calL |\p_x|^{1/2}v=\lambda v.
\]
Upon introducing the new operators
\begin{equation}
\label{a:10}
\calJ\coloneqq\p_x |\p_x|^{-1},\quad
\tL\coloneqq|\p_x|^{1/2} \cl |\p_x|^{1/2},
\end{equation}
we now see that \eref{5} for $u\in L^2(\R)$ can be rewritten as
\begin{equation}\label{20}
\calJ\tL v= \la v,\quad v\in\dot{H}^{-1/2}(\R)\cap H^{s+1/2}(\R).
\end{equation}
 
Consider the operator $\tL$.  Clearly, while  \eref{a:10}   specifies the
action of $\tL$ on smooth vectors, it does not address the important issue of
whether or not $\tL$ is  self-adjoint\footnote{even though it is clearly a
symmetric operator}. For this, one needs to specify a domain. We would like
to point out that there are several (potentially different ways) to obtain a
self-adjoint extension. For the purposes of this section, we proceed in a
canonical way, by building the Friedrich's extension. We will however give a
more direct  construction in \autoref{Sec:4.1}. We follow the arguments in
\citep[Theorem~VIII.15]{reed:fa80}.  More concretely,  consider the bilinear
form
\[
\rmq(f,g)\coloneqq\langle |\partial_x|^{1/2} \calL(|\partial_x|^{1/2} f),   g\rangle =
\dpr{\cl |\partial_x|^{1/2} f}{|\partial_x|^{1/2} g}.
\]
According to \citep[Theorem~VIII.15]{reed:fa80}, if we show that the
quadratic form $\rmq$ is semi-bounded (that is $\rmq(f,f)\geq -M \|f\|^2$ for
some $M$), then $\rmq$ is the quadratic form of an unique self-adjoint
operator, the Friedrich's extension, which we call again $\tL$, with domain
$D(\tL)=\{f\in H^{s+1}(\R): \tL f\in L^2(\R)\}\subset H^{s+1}(\R)$. Let
$\{f_j\}_{j=1}^N$ be a normalized basis of the finite-dimensional negative
subspace of $\cl$, i.e.  $\cl f_j = -\mu_j^2 f_j, j=1, \ldots, N$. In order
to  show the semi-boundedness of $\rmq$,   decompose
$$
|\partial_x|^{1/2} f=h+\sum_{j=1}^N \dpr{|\partial_x|^{1/2} f}{f_j} f_j=
 h+\sum_{j=1}^N \dpr{ f}{|\partial_x|^{1/2}  f_j} f_j,
$$
where $\dpr{\cl h}{h}\geq 0$, since $h\in\Span[f_1, \dots, f_N]^\perp$.
 We have that
$$
\rmq(f,f)=\dpr{\cl |\partial_x|^{1/2} f}{|\partial_x|^{1/2} f}=
\dpr{\cl h}{h}-\sum_{j=1}^N \mu_j^2 \dpr{ f}{|\partial_x|^{1/2}  f_j}^2\geq -
M\|f\|^2,
$$
where $M=N \sup_{j\in [1,N]} (\mu_j^2\|f_j\|_{H^{1/2}}^2)$. Thus, we have
constructed the Friedrich's extension of $\tL$ by virtue of
\citep[Theorem~VIII.15]{reed:fa80}.

Next, several comments are in order regarding the operator $\calJ$. Not only is
this operator skew-symmetric on $L^p(\R)$ for any $1<p<+\infty$, it is a
classical operator, well-studied in the literature; namely, the Hilbert
transform. It can be alternatively defined (on Schwartz functions) via the
formula
\[
\widehat{\calJ f}(\xi)=-\rmi\sgn(\xi) \hat{f}(\xi),
\]
or it can be defined as the singular integral
\[
\calJ f(x)=\f{1}{\pi}\rmp.\rmv. \int_{-\infty}^\infty \f{f(y)}{x-y}\,\rmd y.
\]
Unlike the operator $\partial_x$, the Hilbert transform is a bounded operator
on a variety of function spaces; in particular, on all $L^p(\R)$ for
$1<p<+\infty$. Furthermore, on $L^2(\R)$ it is the case that
$\calJ:L^2(\R)\mapsto L^2(\R)$ is an isometry with $(\calJ)^{-1}=-\calJ$
\citep[Chapter~16.3.2]{lax:fa02}.

We will concentrate our interest on the spectral stability/instability of the
linear system \eref{5}. We say that the linearized problem \eref{5} is
(spectrally) unstable if there is a $\lambda$ with positive real part and a
corresponding function $u\in D(L)\cap H^\infty(\R)\cap \dot{H}^{-1}(\R)$ so
that \eref{5} is satisfied in classical sense. Otherwise, the problem is
spectrally stable. Clearly, spectral instability/stability is equivalent to
the existence (non-existence, respectively) of solutions $v$ to \eref{20}
with $\Re\la>0$. In conclusion, the eigenvalue problem \eref{20} is the
correct one to study in order to apply the previous Hamiltonian-Krein
(instability) index theorems. The application of these theories will require
a careful study of the operator $\tL$.

%
%

\subsection{Relation between the point spectrums of $\calL$ and $\tL$}
 
Before we relate the negative spectrum of the sandwiched operator $\tL$ to
that of $\calL$, we must first understand the kernel of $\tL$.

\begin{lemma}\label{l:kertL}
Regarding the operator $\tL$ we have that $\dim[\ker(\tL)]=1$ with
$\ker(\tL)=\Span\{|\partial_x|^{-1/2}\psi_0\}$.
\end{lemma}

\begin{proof}
Since $\psi_0\in\dot{H}^{-1}(\R)\cap H^\infty(\R)$, it is the case that
$|\partial_x|^{-1/2}\psi_0\in\dot{H}^{-1/2}(\R)\cap H^\infty(\R)$. Since
\[
\calL|\partial_x|^{1/2}\left(|\partial_x|^{-1/2}\psi_0\right)=\calL\psi_0=0,
\]
it is then clear  that $\dim[\ker(\tL)]\ge1$. In order to determine if the
kernel is any larger, consider $\tL u=0$ as an equality of $L^2(\R)$
functions. Testing this equation against all functions $v\in H^{1/2}(\R)$
yields
\[
\langle\tL u,v\rangle=\langle|\partial_x|^{1/2}\calL|\partial_x|^{1/2}u,v\rangle
=\langle\calL|\partial_x|^{1/2}u,|\partial_x|^{1/2}v\rangle=0.
\]
Because of the density \autoref{le:dense} we can rewrite the above as
\[
\langle\calL|\partial_x|^{1/2}u,w\rangle=0,\quad w\in L^2(\R).
\]
Consequently, it must be the case that (as an equality of $L^2(\R)$
functions)
\[
\calL|\partial_x|^{1/2}u=0\quad\Rightarrow\quad u=C|\partial_x|^{-1/2}\psi_0.
\]
The desired conclusion has now been achieved.
\end{proof}


Now that we see the kernel of the sandwiched operator is no larger than the
kernel of the original operator, the next thing to be understood is the
generalized kernel of $\calJ\tL,\,\gker(\calJ\tL)$.

\begin{lemma}\label{l:gkertL}
Suppose that $\psi_0\in\dot{H}^{-1}(\R)$ satisfies
\[
\langle\calL^{-1}(\partial_x^{-1}\psi_0),\partial_x^{-1}\psi_0\rangle\neq0.
\]
The generalized kernel is then given by
\[
\gker(\calJ\tL)=\Span\{|\partial_x|^{-1/2}\psi_0,
|\partial_x|^{-1/2}\calL^{-1}\partial_x^{-1}\psi_0\}.
\]
\end{lemma}

\begin{remark}
Since $\calL$ has a nontrivial kernel, it is not clear that the expression
$\calL^{-1}(\partial_x^{-1}\psi_0)$ is valid. Since $\partial_x^{-1}$ is a
skew-symmetric operator, it is the case that
$\langle\partial_x^{-1}\psi_0,\psi_0\rangle=0$. The fact that $\calL$ is
self-adjoint, and the additional fact that
$\partial_x^{-1}\in\ker(\calL)^\perp$, then tells us that the expression
makes sense.
\end{remark}

\begin{proof}
Since $\calJ$ has bounded inverse, we know from \autoref{l:kertL} that
$\ker(\calJ\tL)=\Span\{|\partial_x|^{-1/2}\psi_0\}$. The first element in the
generalized kernel is then found by solving
\[
\calJ\tL u=|\partial_x|^{-1/2}\psi_0\quad\Rightarrow\quad
\tL u=|\partial_x|^{1/2}\partial_x^{-1}\psi_0.
\]
Since $\psi_0\in\dot{H}^{-1}(\R)$, the expression on the right makes sense.
We would like to begin to use the Fredholm solvability theory at this point,
but unfortunately the fact that the origin is not necessarily isolated from
the (essential) spectrum of $\tL$ means that this is not possible. However,
the form of $\tL$ means that the above is equivalent to
\[
|\partial_x|^{1/2}\calL|\partial_x|^{1/2}u=|\partial_x|^{1/2}\partial_x^{-1}\psi_0\quad\Rightarrow
\quad\calL|\partial_x|^{1/2}u=\partial_x^{-1}\psi_0.
\]
The equality on the right follows from the fact that if $|\partial_x|^{1/2}
G=0$ for an $L^2$ function $G$ (in the sense of distributions), then $G=0$.
Since $\partial_x^{-1}\psi_0\in\ker(\calL)^\perp$, by the Fredholm
solvability theory the above has a solution. The second element in the Jordan
chain is given by
\[
u=|\partial_x|^{-1/2}\calL^{-1}\partial_x^{-1}\psi_0.
\]

The result is proven once it is shown that the Jordan chain is no longer.
Upon continuing we see that the next element in the Jordan chain, if it
exists, is found by solving
\[
\calJ\tL u=|\partial_x|^{-1/2}\calL^{-1}\partial_x^{-1}\psi_0\quad\Rightarrow\quad
\calL|\partial_x|^{1/2}u=\partial_x^{-1}\calL^{-1}\partial_x^{-1}\psi_0.
\]
The Fredholm solvability theory requires that
\[
0=\langle\partial_x^{-1}\calL^{-1}\partial_x^{-1}\psi_0,\psi_0\rangle
=-\langle\calL^{-1}\partial_x^{-1}\psi_0,\partial_x^{-1}\psi_0\rangle.
\]
By assumption this equality cannot hold, which completes the proof.
\end{proof}

Now that the structure of the kernel is well-understood (a crucial ingredient
in the index theories), we now turn to the problem of the negative index for
the operator $\tL$, say $\rmn(\tL)$. In general, we let $\rmn(\calS)$ denote
the number of negative eigenvalues (counting multiplicity) of the
self-adjoint operator $\calS$. Recalling the assumption (b) on the essential
spectrum of $\calL$, we have the following:

\begin{lemma}\label{prop:spec}
Assume that $\sigma_{\ess}(\tL)\setminus\subset [0,+\infty)$. If
$\rmn(\calL)<+\infty$, then $\rmn(\calL)=\rmn(\tL)$.
\end{lemma}

\begin{proof}
Let $\rmn(\calL)=N$, let
$-\lambda_N^2\le-\lambda_{N-1}^2\le\cdots\le-\lambda_1^2<0$ denote the
negative eigenvalues, and let $f_1,\dots,f_N$ denote the corresponding
eigenfunctions. As a consequence of the Courant max/min principle it is known
that
\[
f^\perp\in\Span\{f_1,\dots,f_N\}^\perp\quad\Rightarrow\quad\langle\calL f^\perp,f^\perp\rangle\ge0.
\]
Set $g_j=|\partial_x|^{1/2}f_j$, and let
$g^\perp\in\Span\{g_1,\dots,g_N\}^\perp$ be given. For each $j=1,\dots,N$ we
have
\[
\langle|\partial_x|^{1/2}g^\perp,f_j\rangle=
\langle g^\perp,g_j\rangle=0,
\]
so that $|\partial_x|^{1/2}g^\perp\in\Span\{f_1,\dots,f_N\}^\perp$.
Consequently, we have that
\[
\langle\calL|\partial_x|^{1/2}g^\perp,|\partial_x|^{1/2}g^\perp\rangle\ge0\quad\Rightarrow\quad
\langle\tL g^\perp,g^\perp\rangle\ge0.
\]
In other words, the negative subspace of $\tL$, i.e., the subspace of $\tL$
which corresponds to the negative eigenvalues of $\tL$, must be a subspace of
the negative subspace of $\calL$. In conclusion, we have that $\rmn(\tL)\le
N$.

Now that it is known that $\rmn(\tL)$ is finite, assume that $\rmn(\tL)=M$.
Equality of the two indices for $M=0$ follows immediately from
\autoref{le:dense}, so assume $M\geq 1$. We first show that all
eigenfunctions of $\tL$ corresponding to non-zero eigenvalues belong to
$\dot{H}^{-1/2}(\R)$. Indeed, let $\mu\neq 0$ be an eigenvalue, with
eigenfunction $f$, so that
\[
\mu f= \tL f=|\p_x|^{1/2}\left(\calL |\p_x|^{1/2} f\right).
\]
Since $|\partial_x|^{1/2}:L^2(\R)\mapsto\dot{H}^{-1/2}(\R)$ is an isometry,
the result now follows.

Next, let $f_1,\dots,f_M$ be the normalized eigenfunctions of $\tL$ which
correspond to the negative eigenvalues $-\mu_M^2\leq-\mu_{N-1}^2\le\cdots
\leq -\mu_1^2<0$. For $j=1,\dots,M$ set $g_j =|\p_x|^{-1/2} f_j\in L^2(\R)$
(in fact $\|g_j\|_{L^2}=\|f_j\|_{\dot{H}^{-1/2}}$)  and fix $g\in\Span\{g_1,
\ldots, g_M\}^\perp$ so that $\|g\|_{H^s}\leq 1$. For $0<\ve<<1$, we have
\[
\dpr{\calL g}{g}=\dpr{\calL g_{>\ve}}{g_{>\ve}}+2 \dpr{\calL g_{>\ve}}{g_{\leq \ve}}+
\dpr{\calL g_{\leq \ve}}{g_{\leq \ve}}.
\]
By using Cauchy-Schwartz the latter two terms can be bounded via
$$
2|\dpr{L g_{>\ve}}{g_{\leq \ve}}|+
|\dpr{L g_{\leq \ve}}{g_{\leq \ve}}|\leq C\|g\|_{H^s}\|g_{\leq \ve}\|_{L^2}=C \|g_{\leq \ve}\|_{L^2},
$$
where we have used that\footnote{Note that we assume that $\cl: D(\cl)\subset
H^s\to L^2$ and hence the estimate $\|\cl g\|_{L^2}\leq C\|g\|_{H^s}$} $\|\cl
g\|_{L^2}\leq C\|g\|_{H^s}$ Regarding the first term $\dpr{\calL
g_{>\ve}}{g_{>\ve}}$,  write
$$
\dpr{\calL g_{>\ve}}{g_{>\ve}}=\dpr{\calL |\p_x|^{1/2} |\p_x|^{-1/2} g_{>\ve}}{|\p_x|^{1/2} |\p_x|^{-1/2} g_{>\ve}}=\dpr{\tL |\p_x|^{-1/2} g_{>\ve}}{|\p_x|^{-1/2} g_{>\ve}}
$$
Projecting $|\p_x|^{-1/2} g_{>\ve}$ over the spectrum of $\tL$ yields
\[
|\p_x|^{-1/2} g_{>\ve}=h_\epsilon+\sum_{j=1}^M\langle|\p_x|^{-1/2} g_{>\ve},f_j\rangle f_j,
\]
where $\langle\tL h_\epsilon,h_\epsilon\rangle\ge0$ and
$h_\epsilon\in\Span\{f_1,\dots,f_M\}^\perp$. Since
$$
0=\dpr{g}{g_j}=\dpr{|\p_x|^{-1/2} f_j}{g_{>\ve}}+\dpr{g_j}{g_{\leq \ve}}=
\dpr{ f_j}{|\p_x|^{-1/2}  g_{>\ve}}+\dpr{g_j}{g_{\leq \ve}},
$$
we can rewrite the above expansion as
$$
|\p_x|^{-1/2}  g_{>\ve}=h_\ve- \sum_{j=1}^M \dpr{g_j}{g_{\leq \ve}} f_j.
$$
It then follows that
$$
\dpr{\tL |\p_x|^{-1/2} g_{>\ve}}{|\p_x|^{-1/2} g_{>\ve}}=
\dpr{\tL h_\ve}{h_\ve}-\sum_{j=1}^M \mu_j^2|\dpr{g_j}{g_{\leq \ve}}|^2.
$$
Using the definition of $\tL$ we can rewrite the above as
\[
\dpr{\calL g_{>\ve}}{g_{>\ve}}=
\dpr{\tL h_\ve}{h_\ve}-\sum_{j=1}^M \mu_j^2|\dpr{g_j}{g_{\leq \ve}}|^2.
\]
Again using Cauchy-Schwartz we have that $|\dpr{g_j}{g_{\leq \ve}}|\leq
\|g_j\|_{L^2} \|g_{\leq \ve}\|_{L^2}\leq C \|g_{\leq \ve}\|_{L^2}$, where \\  $C=\sup_{j\in [1,N]} \|f_j\|_{\dot{H}^{-1/2}}$.
It follows that
\[
\dpr{L g_{>\ve}}{g_{>\ve}}\geq \dpr{\tL h_\ve}{h_\ve}-C\|g_{\leq \ve}\|_{L^2}.
\]
In addition, note that by Cauchy-Schwartz
$$
\dpr{\cl g_{\leq \ve}}{ g_{\leq \ve}}\leq \| \cl g_{\leq \ve}\|_{L^2} \| g_{\leq \ve}\|_{L^2}\leq C \|g_{\leq \ve}\|_{H^s}\| g_{\leq \ve}\|_{L^2}\leq C \| g_{\leq \ve}\|_{L^2}^2.
$$
Putting everything together yields
\[
\dpr{L g}{g}\geq  -C\|g_{\leq \ve}\|_{L^2}(1+\|g_{\leq \ve}\|_{L^2}).
\]
Since $\epsilon>0$ is arbitrary and $\lim_{\epsilon\to 0}\|g_{\leq
\ve}\|_{L^2}=0$,  it must then be the case that $\dpr{\calL g}{g}\geq 0$.
This inequality implies that $\rmn(\calL)\leq M=\rmn(\tL)$. The proof is now
complete.
\end{proof}

\subsection{An example}

In the previous section, we have considered the theoretical relation between
the spectral properties of $\cl$ and $\tL$. We would like now to explore it
further for a specific example.

\begin{proposition}\label{prop:sch}
Let $V:\R\to \R$ be $C^1$ smooth and (sufficiently) decaying potential.
Consider the corresponding Schr\"odinger operator $\cl\coloneqq-\p_x^2+c -
V$, with $c>0$ and  $D(\cl)=H^2(\R)$. The sandwiched operator
$\tL\coloneqq|\p_x|^{1/2} \cl |\p_x|^{1/2}$ with domain $D(\tL)=H^3(\R)$ is
self-adjoint, and moreover $\si_{\ess}(\tL)=[0, \infty)$, while
$\rmn(\tL)=\rmn(\cl)<\infty$.
\end{proposition}

\begin{proof}
One could argue that the Friedrich's extension of $\tL$ is self-adjoint,
after which, one  will need to identify the domain as $H^3(\R)$. We will
instead follow a more direct route in constructing a self-adjoint extension
of the symmetric operator $\tL$. To that end, let
$$
\tL=|\p_x|^{1/2} \cl |\p_x|^{1/2} = -|\p_x|\p_x^2+|\p_x|^{1/2} V |\p_x|^{1/2}=
|\p_x|^3+ |\p_x|^{1/2} V |\p_x|^{1/2}=:\tL_0+\calK^\diamond,
$$
where $D(\tL_0)=H^3(\R),\,D(\calK^\diamond)=H^1(\R)$. Clearly, with these
assignments, $\tL_0$ is self-adjoint, while $\calK^\diamond$ is a symmetric
operator. Furthermore, $\si(\tL_0)=\si_{\ess}(\tL_0)=[0, \infty)$.

According to  the Weyl's essential spectrum theorem (and more specifically
\citep[Corollary 2, page 113]{reed:aoo78}), we may conclude that
$\tL=\tL_0+\calK^\diamond$ is self-adjoint and $\si_{ess}(\tL)
=\si_{\ess}(\tL_0)=[0, \infty)$ provided we can establish that
$\calK^\diamond$ is a relatively compact perturbation of $\tL_0$. This
amounts to showing that
$$
|\p_x|^{1/2} V |\p_x|^{1/2} (\rmi+|\p_x|^3)^{-1}:L^2(\R)\mapsto L^2(\R)\,\,\mathrm{is\,\,compact}.
$$
Since it is clear that $|\p_x|^{1/2} V |\p_x|^{1/2} (\rmi+|\p_x|^3)^{-1}\in
B(L^2(\R))$, it will suffice (by Relich's criteria) to establish
\begin{eqnarray}
\label{50}
& & |\p_x|^{1/2} V |\p_x|^{1/2} (i+|\p_x|^3)^{-1}: L^2(\R)\to \dot{H}^{1/2}(\R) \\
\label{60}
& &  ||\p_x|^{1/2} V |\p_x|^{1/2} (\rmi+|\p_x|^3)^{-1} f (x)| \leq  \f{C}{\sqrt{|x|}}
\|f\|_{L^2},\quad |x|>>1.
\end{eqnarray}

For the proof of \eref{50}, we have
\[
\begin{split}
\||\p_x|^{1/2} V |\p_x|^{1/2} (\rmi+|\p_x|^3)^{-1} f\|_{\dot{H}^{1/2}(\R)} &=
\|\p_x[V |\p_x|^{1/2} (\rmi+|\p_x|^3)^{-1} f]\|_{L^2(\R)}\\
&\leq \|V'\|_{L^\infty}
\| |\p_x|^{1/2} (\rmi+|\p_x|^3)^{-1} f\|_{L^2}+ \\
&\qquad\|V\|_{L^\infty} \||\p_x|^{3/2} (\rmi+|\p_x|^3)^{-1} f\|_{L^2}\\
&\leq C(\|V'\|_{L^\infty}+ \|V\|_{L^\infty})\|f\|_{L^2}.
\end{split}
\]
For the  proof of \eref{60}, first fix $x$ with $|x|>>1$. Denoting
$h(x)=|\p_x|^{1/2} (\rmi+|\p_x|^3)^{-1} f$ and $G(x)=V(x) h(x)$,  write
\[
\begin{split}
|\p_x|^{1/2} V |\p_x|^{1/2} (\rmi+|\p_x|^3)^{-1} f (x)&=
\sqrt{2\pi} \int |\xi|^{1/2} \hat{G}(\xi)\rme^{2\pi\rmi x \xi}\, \rmd x\\
&= \sqrt{2\pi} \sum_{k=-\infty}^\infty 2^{k/2} \int \tilde{\vp}(2^{-k} \xi)
\hat{G}(\xi)\rme^{2\pi\rmi x \xi}\,\rmd x,
\end{split}
\]
where $\tilde{\vp}(z)=|z|^{1/2} \vp(z)$. For the portion of the sum $k$ with
$2^k <|x|^{-1}$, we have
$$
\sum_{k:2^k <|x|^{-1} } 2^{k/2} \int \tilde{\vp}(2^{-k} \xi) \hat{G}(\xi)\rme^{2\pi\rmi x \xi}\,\rmd x
\leq C|x|^{-1/2} \sup_k \| G_{\sim 2^k}\|_{L^\infty}\leq C |x|^{-1/2}\|G\|_{L^\infty}\leq C |x|^{-1/2} \|V\|_{L^\infty} \|h\|_{L^\infty}.
$$
By Sobolev embedding, for any $m>1/2$, $ \|h\|_{L^\infty}\leq C_m
\|h\|_{H^{m}} \leq C \|f\|_{L^2}. $ Regarding the case of $k$ with $2^k \geq
|x|^{-1}$ we integrate by parts to get
$$
\int \tilde{\vp}(2^{-k} \xi) \hat{G}(\xi)\rme^{2\pi\rmi x \xi}\,\rmd x=\f{\rmi}{2\pi x}
\left( 2^{-k} \int  \tilde{\vp}'(2^{-k} \xi) \hat{G}(\xi)\rme^{2\pi\rmi x \xi}\,\rmd x+
\int \tilde{\vp}(2^{-k} \xi) \hat{G}'(\xi)\rme^{2\pi\rmi x \xi}\,\rmd x\right).
$$
The first term is estimated via
\begin{eqnarray*}
\left|\sum_{k: 2^k \geq |x|^{-1}} \f{2^{-k/2}}{2\pi x} \int  \tilde{\vp}'(2^{-k} \xi) \hat{G}(\xi)
\rme^{2\pi\rmi x \xi}\,\rmd x\right|\leq C |x|^{-1/2} \sup_k \| G_{\sim 2^k}\|_{L^\infty},
\end{eqnarray*}
whence the estimate finishes as the one a few lines up. Finally, upon
introducing the smooth function $\Phi(z)=\tilde{\vp}(z)/z$, rewrite
\[
\int \tilde{\vp}(2^{-k} \xi) \hat{G}'(\xi)\rme^{2\pi\rmi x \xi}\,\rmd x=2^{-k} \int
\Phi(2^{-k} \xi) [\xi \hat{G}'(\xi)]\rme^{2\pi\rmi x \xi}\,\rmd x,
\]
so that
$$
\left|\sum_{k: 2^k \geq |x|^{-1}} \f{2^{-k/2}}{2\pi x} \int \Phi(2^{-k} \xi) [\xi \hat{G}'(\xi)]
\rme^{2\pi\rmi x \xi}\,\rmd x\right| \leq C |x|^{-1/2}
 \sup_k \|H_{\sim 2^k}\|_{L^\infty}\leq  C |x|^{-1/2}
   \|H\|_{L^\infty}
$$
where $H$ is defined through its Fourier transform, $\hat{H}(\xi)\coloneqq\xi
\hat{G}'(\xi)$. From the formula for $H$ one can easily identify it; namely,
$H(x)=c_0 \p_x( x V(x) h(x))$ for some constant $c_0$. Finally, by H\"older's
inequality and Sobolev embedding we conclude with
$$
\|H\|_{L^\infty}\leq C(\|x V(x)\|_{L^\infty}+ \|x V'(x)\|_{L^\infty}+ \|V\|_{L^\infty})
(\|h\|_{L^\infty}+ \|h'\|_{L^\infty})\leq C_V \|f\|_{L^2}.\qedhere
$$
\end{proof}

\section{The Hamiltonian-Krein index theorem}\label{s:hki}

  We are now ready to apply the instability index formula of
\citep{kapitula:cev04,kapitula:ace05}. For the reformulated eigenvalue
problem \eref{20} let $k_\rmr$ denote the number of real-valued and positive
eigenvalues (counting multiplicity), and let $k_\rmc$ be the number of
complex-valued eigenvalues with positive real part. Since the imaginary part
of $\calL$ satisfies $\Im(\calL)=0$, it is the case that $k_\rmc$ is an even
integer. Finally, for (potentially embedded) purely imaginary eigenvalues
$\lambda\in\rmi\R$, let $E_\lambda$ denote the corresponding eigenspace. The
negative Krein index of the eigenvalue is given by
\[
k_\rmi^-(\lambda)=\rmn(\langle\tL|_{E_\lambda}u,u\rangle),
\]
and the total negative Krein index is given by
\[
k_\rmi^-=\sum_{\lambda\in\rmi\R} k_\rmi^-(\lambda).
\]
Here we use the notation $\calS|_E=P_E\calS P_E$, where $P_E$ is the
orthogonal projection onto the subspace $E$. It is the case that $k_\rmi^-$
is also an even integer. The Hamiltonian-Krein index is defined by
\[
K_{\Ham}\coloneqq k_\rmr+k_\rmc+k_\rmi:
\]
the index counts the total number of (potentially) unstable eigenvalues.

\begin{remark}
The negative Krein index of the operator is defined in terms of the the
eigenvalue problem \eref{20}. Using the definition of $\tL$ it can be
rewritten as
\[
k_\rmi^-(\lambda)=
\rmn(\langle\calL|_{E_\lambda}|\partial_x|^{1/2}u,|\partial_x|^{1/2}u\rangle).
\]
In terms of the original eigenvalue problem \eref{5}, upon using the
transformation that moved the first to the second yields that in terms of the
original variables, the negative Krein index can be rewritten as
\[
k_\rmi^-(\lambda)=\rmn(\langle\calL|_{E_\lambda}u,u\rangle).
\]
This is the expected definition, and the one that would have been used if the
operator $\partial_x$ had bounded inverse.
\end{remark}

For the eigenvalue problem \eref{20} the previous index theory relates
$K_{\Ham}$ to the finite number $\rmn(\tL|_{S^\perp})$, where $S$ is some
finite-dimensional subspace. In order to apply this theory, however, it is
not only necessary that $\calJ$ have bounded inverse, but that
$\sigma_{\ess}(\tL)$ be uniformly bounded away from the origin.
Unfortunately, this technical assumption is no longer necessarily valid, as
in the applications it will generically be the case that the essential
spectrum touches the origin, i.e., $\sigma_{\ess}(\tL)=[0,+\infty)$ (e.g.,
see \autoref{prop:sch}). We overcome this technical difficulty with the
following argument.


For $0<\epsilon\ll1$ consider the sandwiched operator
\[
\tL_\epsilon\coloneqq(-\partial_x^2+\epsilon^2)^{1/4}
\calL(-\partial_x^2+\epsilon^2)^{1/4}.
\]
Note that $\tL=\tL_0$, and moreover $\tL_\epsilon\to\tL$ as $\epsilon\to0^+$
in the weak operator topology. That is, for each pair of  test functions
$\chi, \psi$,
\begin{equation}
\label{o:1}
\lim_{\epsilon\to 0} \dpr{\tL_\epsilon \chi}{\psi}= \dpr{\tL  \chi}{\psi}
\end{equation}
Since the operator $(-\partial_x^2+\epsilon^2)^{1/4}$ is invertible for any
$\epsilon>0$,\footnote{In fact, one might argue as in \autoref{prop:spec}
that for each $\epsilon>0,\,\rmn(\tL_\epsilon)=\rmn(\cl)$}
\[
\rmn(\tL_\epsilon)=\rmn(\calL);
\]
thus, by using \autoref{prop:spec} we have the double equality
\begin{equation}\label{e:340}
\rmn(\tL_\epsilon)=\rmn(\calL)=\rmn(\tL).
\end{equation}
Regarding the essential spectrum, we would like to say that it is pushed off
the origin and becomes $\sigma_{\ess}(\tL_\epsilon)\subseteq [\ka^2
\epsilon,+\infty)$. This is indeed the case, but it needs some justification.
\begin{proposition}
\label{prop:109} The essential spectrum of the sandwiched operator satisfies
$$
\sigma_{\ess}(\tL_\epsilon)\subseteq [\ka^2 \epsilon,+\infty).
$$
\end{proposition}

\begin{proof}
Recall that by assumption $\cl=\calL_0+K$,  where $\calK$ is relatively
compact perturbation of $\calL_0\geq \ka^2\calI$. In particular,  $\calL_0$
is invertible and $\calK\calL_0^{-1}$ is a compact operator.  Denote
$$
\tilde{\calL}_\epsilon=(-\partial_x^2+\epsilon^2)^{1/4}\calL_0 (-\partial_x^2+\epsilon^2)^{1/4}; \ \
\tilde{\calK}_\epsilon=(-\partial_x^2+\epsilon^2)^{1/4}
\calK (-\partial_x^2+\epsilon^2)^{1/4},
$$
so that $\tL_\epsilon= \tilde{\calL}_\epsilon+\tilde{\calK}_\epsilon$. By
applying the Fourier transform, we see that
$$
\si(\tilde{\calL}_\epsilon)=\si_{\ess}(\tilde{\calL}_\epsilon)=\mathrm{Range}[\xi\mapsto (4\pi^2\xi^2+\epsilon^2)^{1/4} q(\xi) (4\pi^2\xi^2+\epsilon^2)^{1/4}]\subseteq [\ka^2\epsilon, \infty),
$$
since by assumption $q(\xi)\geq \ka^2$. Thus,  it remains to show
\begin{equation}
\label{o:2}
 \si_{\ess}(\tilde{\calL}_\epsilon+\tilde{\calK}_\epsilon)= \si_{\ess}(\tilde{\calL}_\epsilon).
\end{equation}
In order to show \eref{o:2}, we will verify that
$(\tilde{\calL}_\epsilon+\tilde{\calK}_\epsilon+\rmi)^{-1}-(\tilde{\calL}_\epsilon+\rmi)^{-1}$
is a compact operator, whence the result will follow from a standard lemma in
spectral theory \cite[Corollary 1,p. 113]{reed:aoo78}. Indeed,
$$
\tilde{\calL}_\epsilon^{-1} \tilde{\calK}_\epsilon \tilde{\calL}_\epsilon^{-1}=
(-\p_x^2+\epsilon^2)^{-1/4}\calL_0^{-1}\calK\calL_0^{-1} (-\p_x^2+\epsilon^2)^{-1/4}
$$
is compact, since $\calK\calL_0^{-1}$ is compact by assumption. A multiple
application of the  resolvent identity yields the formula
$$
(\tilde{\calL}_\epsilon+\tilde{\calK}_\epsilon+\rmi)^{-1} \tilde{\calK}_\epsilon (\tilde{\calL}_\epsilon+\rmi)^{-1}=
[(\tilde{\calL}_\epsilon^{-1}-
(\tilde{\calL}_\epsilon+\tilde{\calK}_\epsilon+\rmi)^{-1}(\tilde{\calK}_\epsilon+\rmi)\tilde{\calL}_\epsilon^{-1}] \tilde{\calK}_\epsilon[(\tilde{\calL}_\epsilon^{-1}+\rmi \tilde{\calL}_\epsilon^{-1}
(\tilde{\calL}_\epsilon+\rmi)^{-1}],
$$
which implies that $(\tilde{\calL}_\epsilon+\tilde{\calK}_\epsilon+\rmi)^{-1}
\tilde{\calK}_\epsilon (\tilde{\calL}_\epsilon+\rmi)^{-1}$ is compact as
well. From the resolvent identity again,
$$
(\tilde{\calL}_\epsilon+\tilde{\calK}_\epsilon+\rmi)^{-1}-(\tilde{\calL}_\epsilon+\rmi)^{-1}=
-(\tilde{\calL}_\epsilon+\tilde{\calK}_\epsilon+\rmi)^{-1} \tilde{\calK}_\epsilon (\tilde{\calL}_\epsilon+\rmi)^{-1},
$$
which implies the compactness of
$(\tilde{\calL}_\epsilon+\tilde{\calK}_\epsilon+\rmi)^{-1}-(\tilde{\calL}_\epsilon+\rmi)^{-1}$,
and hence \eref{o:2} is established.
\end{proof}

For Hamiltonian eigenvalue problems the Hamiltonian-Krein index theory can be
applied if the skew-operator has a bounded inverse, and if the self-adjoint
operator (a) has a finite number of negative eigenvalues, and (b) has an
essential spectrum which is bounded away from the origin. Thus, this theory
is applicable for the operator $\calJ\tL_\epsilon$ for any $\epsilon>0$. We
now wish to establish the index for the one-parameter family of eigenvalue
problems given by
\begin{equation}\label{e:341}
\calJ\tL_\epsilon u=\lambda u\quad\Rightarrow\quad\tL_\epsilon u=\lambda\calJ^{-1}u.
\end{equation}
Afterwards, we will use a limiting argument to establish the index for the
original problem \eref{5}.

Since $\ker(\tL_\epsilon)=\Span\{(-\partial_x^2+\epsilon^2)^{-1/4}\psi_0\}$,
and since $\tL_\epsilon$ is self-adjoint, for $\lambda\neq0$ the eigenvalue
problem \eref{e:341} can by the Fredholm alternative have a solution if and
only if
\[
\langle\calJ^{-1}u,(-\partial_x^2+\epsilon^2)^{-1/4}\psi_0\rangle=0\quad\Rightarrow\quad
\langle u,(-\partial_x^2+\epsilon^2)^{-1/4}|\partial_x|\partial_x^{-1}\psi_0\rangle=0.
\]
Thus, upon setting
$S_\epsilon\coloneqq\Span\{(-\partial_x^2+\epsilon^2)^{-1/4}|\partial_x|\partial_x^{-1}\psi_0\}$,
the search for nonzero eigenvalues is accomplished by considering the
constrained problem
\begin{equation}\label{e:342}
\calJ\tL_\epsilon u=\lambda u,\quad u\in S_\epsilon^\perp.
\end{equation}

As a consequence of \citep{kapitula:cev04,kapitula:ace05} it is true that for
$\epsilon>0$ the Hamiltonian-Krein index associated with the eigenvalue
problem \eref{e:342} satisfies
\[
K_{\Ham}^\epsilon=\rmn(\tL_\epsilon|_{S^\perp}).
\]
Using, e.g., \citep{kapitula:sif12}, we have that the negative index on the
right satisfies
\[
\begin{split}
\rmn(\tL_\epsilon|_{S^\perp})&=\rmn(\tL_\epsilon)-
\rmn\left(\langle(\tL_\epsilon)^{-1}(-\partial_x^2+\epsilon^2)^{-1/4}|\partial_x|\partial_x^{-1}\psi_0,
(-\partial_x^2+\epsilon^2)^{-1/4}|\partial_x|\partial_x^{-1}\psi_0\rangle\right)\\
&=\rmn(\calL)-
\rmn\left(\langle(\tL_\epsilon)^{-1}(-\partial_x^2+\epsilon^2)^{-1/4}|\partial_x|\partial_x^{-1}\psi_0,
(-\partial_x^2+\epsilon^2)^{-1/4}|\partial_x|\partial_x^{-1}\psi_0\rangle\right).
\end{split}
\]
The second equality follows from \eref{e:340}. Regarding the second quantity
on the right, by the definition of $\tL_\epsilon$ we have
\[
\begin{split}
&\langle(\tL_\epsilon)^{-1}(-\partial_x^2+\epsilon^2)^{-1/4}|\partial_x|\partial_x^{-1}\psi_0,
(-\partial_x^2+\epsilon^2)^{-1/4}|\partial_x|\partial_x^{-1}\psi_0\rangle=\\
&\qquad\qquad\langle\calL^{-1}(-\partial_x^2+\epsilon^2)^{-1/2}|\partial_x|\partial_x^{-1}\psi_0,
(-\partial_x^2+\epsilon^2)^{-1/2}|\partial_x|\partial_x^{-1}\psi_0\rangle.
\end{split}
\]
Note that  the expression
$\calL^{-1}[(-\partial_x^2+\epsilon^2)^{-1/2}|\partial_x|\partial_x^{-1}\psi_0]$
makes sense, since
$(-\partial_x^2+\epsilon^2)^{-1/2}|\partial_x|\partial_x^{-1}\psi_0\in
\ker(\calL)^\perp$. Since
$(-\partial_x^2+\epsilon^2)^{-1/2}|\partial_x|\partial_x^{-1}\psi_0\to
\partial_x^{-1}\psi_0$ in $L^2(\R)$ sense, we have
 $$
 \dpr{\calL^{-1}(-\partial_x^2+\epsilon^2)^{-1/2}|\partial_x|\partial_x^{-1}\psi_0}{
(-\partial_x^2+\epsilon^2)^{-1/2}|\partial_x|\partial_x^{-1}\psi_0}\to
\dpr{\calL^{-1} \partial_x^{-1}\psi_0}{\partial_x^{-1}\psi_0}.
 $$
It follows that for all $\epsilon>0$ sufficiently small
\[
\rmn\left(\langle\calL^{-1}(-\partial_x^2+\epsilon^2)^{-1/2}|\partial_x|\partial_x^{-1}\psi_0,
(-\partial_x^2+\epsilon^2)^{-1/2}|\partial_x|\partial_x^{-1}\psi_0\rangle\right)=
\rmn(\langle\calL^{-1}(\partial_x^{-1}\psi_0),\partial_x^{-1}\psi_0\rangle).
\]
Putting this all together, we see that the Hamiltonian-Krein index for
$\epsilon>0$ small satisfies
\[
K_{\Ham}^\epsilon=\rmn(\calL)-
\rmn(\langle\calL^{-1}(\partial_x^{-1}\psi_0),\partial_x^{-1}\psi_0\rangle).
\]
Since the quantity on the right is $\epsilon$-independent, and since the
index is integer-valued, we can take the limit as $\epsilon\to0^+$ and
conclude that for the original eigenvalue problem \eref{20},
\[
K_{\Ham}^0=\rmn(\calL)-
\rmn(\langle\calL^{-1}(\partial_x^{-1}\psi_0),\partial_x^{-1}\psi_0\rangle).
\]
In other words, the Hamiltonian-Krein index for the equivalent problem
\eref{20} does not depend at all on the reformulation of the eigenvalue
problem, and can be stated in terms of the original operators $\partial_x$
and $\calL$.

Finally, recall that the reformulated eigenvalue problem assumed that all of
the eigenfunctions resided in $\dot{H}^{-1/2}(\R)$. Since the eigenfunctions
associated with nonzero eigenvalues must reside in $\dot{H}^{-1}(\R)$, the
index for the original eigenvalue problem \eref{5} must be the same as for
the reformulated problem. This allows us to conclude with the following:

\begin{theorem}\label{t:index}
Consider the eigenvalue problem
\[
\partial_x\calL u=\lambda u,\quad u\in L^2(\R),
\]
where the self-adjoint operator $\calL$ satisfies $D(\calL)=H^s(\R)$ for some
$s\geq 0$. Assuming that
\[
\langle\calL^{-1}(\partial_x^{-1}\psi_0),\partial_x^{-1}\psi_0\rangle\neq0,
\]
the Hamiltonian-Krein index satisfies
\[
K_{\Ham}=\rmn(\calL)-
\rmn\left(\langle\calL^{-1}(\partial_x^{-1}\psi_0),\partial_x^{-1}\psi_0\rangle\right).
\]
\end{theorem}

\begin{remark}
The result of \autoref{t:index} is the one that would be expected if the skew
operator $\partial_x$ had a bounded inverse. On the other hand, when
considering the eigenvalue problem on the space of spatially periodic
functions, the index satisfies
\[
K_{\Ham}=\rmn(\calL)-\rmn\left(\vD\right),\quad
\vD=\left(\begin{array}{cc}
\langle\calL^{-1}(\partial_x^{-1}\psi_0),\partial_x^{-1}\psi_0\rangle&
\langle\calL^{-1}(\partial_x^{-1}\psi_0),1\rangle\\
\langle\calL^{-1}(\partial_x^{-1}\psi_0),1\rangle&
\langle\calL^{-1}(1),1\rangle
\end{array}\right)
\]
\citep{deconinck:ots10,bronski:ait11}. The $(1,1)$ term in the matrix $\vD$
is exactly that as seen in the above theorem. In this latter case, the
additional terms in $\vD$ arise from the fact that
$\ker(\partial_x)=\Span\{1\}$: this kernel is not present when considering
the problem on the whole line in the space $L^2(\R)$. In both cases the
eigenfunctions associated with nonzero eigenvalues must have mean zero, but
interestingly enough it is only in the latter case that this restriction
actually has an effect on the index.
\end{remark}

\section{Applications to KdV-like problems}
 
 \label{Sec:4.1}

Consider an evolution equation written in the form
\begin{equation}\label{e:kdv}
\partial_tu=\partial_x\calH'(u),
\end{equation}
where $\calH(u):H^\ell(\R)\mapsto\R$ is some smooth energy functional. It is
assumed that solutions are invariant under spatial translation, and that
there are no other symmetries present in the system. In traveling coordinates
$\xi=x-ct$ the equation can be rewritten as
\[
\partial_tu=\partial_\xi\left[\calH'(u)+cu\right].
\]
For a given $a\in\R$ the solitary wave solutions $U_c$ will satisfy
\begin{equation}\label{e:kdvs}
\calH'(U_c)=-cU_c+a.
\end{equation}
It will be assumed that these solutions are smooth in the wave-speed over
some nonempty interval. It will further be assumed that $U_c\in H^k(\R)$ for
some $k\ge1$ and all $c$.

The linearized eigenvalue problem is given by
\begin{equation}\label{e:kdveval}
\partial_\xi\calL u=\lambda u,\quad
\calL\coloneqq\calH''(U_c)+c.
\end{equation}
Since solutions are invariant under spatial translation, it is the case that
\[
\calL(\partial_\xi U_c)=0.
\]
By assumption it is clearly the case that $\partial_\xi
U_c\in\dot{H}^{-1}(\R)$. Now, differentiating the existence problem
\eref{e:kdvs} with respect to $c$ yields
\[
\calL(\partial_c U_c)=-U_c\quad\Rightarrow\quad
\partial_cU_c=-\calL^{-1}(U_c).
\]
As a consequence, upon making the equivalence that $\psi_0=\partial_\xi U_c$
in the statement of \autoref{t:index}, it is seen that $\p_{\xi}^{-1} \psi_0=U_c$ and hence
\[
\langle\calL^{-1}\partial_\xi^{-1}\psi_0,\partial_\xi^{-1}\psi_0\rangle=
-\langle\partial_cU_c,U_c\rangle=-\frac12\partial_c\langle U_c,U_c\rangle.
\]
Defining $\rmp(a)=1$ for $a>0$ and $\rmp(a)=0$ for $a<0$, the
Hamiltonian-Krein index for the eigenvalue problem \eref{e:kdveval} can then
be stated as:

\begin{theorem}\label{thm:kdvsp}
Consider the KdV-like evolution equation \eref{e:kdv}. Let $U_c$ be a
solitary wave which solves \eref{e:kdvs} and satisfies the properties
proscribed above. Let
\[
\calL\coloneqq\calH''(U_c)+c
\]
be the linearization about the wave in traveling coordinates. Assume that
\[
\ker(\calL)=\Span\{\partial_\xi U_c\},\quad
 \partial_c\langle U_c,U_c\rangle\neq0.
\]
The Hamiltonian-Krein index for the linearized eigenvalue problem
\eref{e:kdveval} is then
\[
K_{\Ham}=\rmn(\calL)-\rmp\left(\partial_c\langle U_c,U_c\rangle\right).
\]
\end{theorem}

\begin{remark}
Since $k_\rmc$ and $k_\rmi^-$ are even integers, the underlying wave will be
spectrally unstable with $k_\rmr\ge1$ if $K_{\Ham}$ is odd.
\end{remark}

\begin{remark}
From \autoref{thm:kdvsp} we (partially) recover the following well-known
result. Suppose that $\rmn(\calL)=1$. If $\partial_c\langle
U_c,U_c\rangle>0$, then under some additional genericity conditions it is known
that the wave is orbitally stable \citep{bona:sai87} (also see
\citep[Chapter~5]{kapitula:ait12}). As we see from \autoref{thm:kdvsp}, under
this scenario $K_{\Ham}=0$, so that all of the spectra is purely imaginary,
and any (embedded) eigenvalues must have positive Krein signature. On the
other hand, if $\partial_c\langle U_c,U_c\rangle<0$, then $K_{\Ham}=1$
implies that $k_\rmr=1$, which corroborates the Evans function calculation of
\citet{pego:eis92,pego:efm93}.
\end{remark}

\subsection{Fractional  KdV equations}\label{fracKDV}

The exact form of $\calL$ in \autoref{thm:kdvsp} is unimportant, as long as
the desired properties hold. Thus, as a generalization of the above we obtain
the following result of \citet[Theorem~1]{lin:son08} concerning the more
general KdV-type equation
\begin{equation}\label{lin:kdv}
\partial_tu-\p_x\left(\cm u-f(u)\right)=0,
\end{equation}
where $f\in C^1$, $f(0)=f'(0)=0$, and the operator $\calM$ is defined through
its Fourier symbol via $\widehat{\cm g}(\xi)=\al(\xi) \hat{g}(\xi)$. Examples
of this type include the KdV equations (and its generalizations) with
$\cm=-\p_x^2$, the Benjamin-Ono equation with $\cm=|\p_x|$, etc. We
henceforth will assume that the multiplier $\al(\xi)$ is a continuous
function of its argument with $\lim_{|\xi|\to \infty} \al(\xi)=\infty$.
 
\begin{corollary}\label{cor:44}
For the generalized KdV-type equation \eref{lin:kdv} assume that the
linearized operator $\cl= \cm+c-f'(U_c)$ satisfies the original assumptions
(a), (b), (c) with
$$
\ker(\cl)= \{\p_\xi U_c\}.
$$
The Hamiltonian-Krein index for the spectral problem $\partial_\xi\calL
u=\lambda u$ is
\begin{equation}
\label{eq:07}
K_{\Ham}=\rmn(\calL)-\rmp\left(\partial_c\langle U_c,U_c\rangle\right).
\end{equation}
In particular,
\begin{itemize}
\item the wave $U_c$ is spectrally unstable if $K_{\Ham}$ is odd, i.e.,
    \begin{itemize}
    \item $\rmn(\cl)$ is even and $\partial_c\langle U_c,U_c\rangle>0$
    \item $\rmn(\cl)$ is odd and $\partial_c\langle U_c,U_c\rangle<0$
    \end{itemize}
\item the wave $U_c$ is spectrally stable if $\rmn(\cl)=1$ and
    $\partial_c\langle U_c,U_c\rangle>0$.
\end{itemize}
\end{corollary}
 
One should compare these results with the corresponding results of
\citet[Theorem~1]{lin:son08} for KdV-type equations. In particular, we do not
require the multiplier $\al(\xi)$ to have any specific form, like
$\al(\xi)\sim |\xi|^m$, but only the natural conditions
\begin{itemize}
\item  $\cm+c\geq \de^2\calI>0$ for some $\de>0$,
\item $f'(U_c)$ is a relatively compact perturbation of $\cm+c$.
\end{itemize}
The condition $\cm+c\geq \de^2\calI$ is satisfied by simply requiring that
$\al(\xi)\geq 0$  and $c>0$, in which case, we may select $\de=c/2$. The
relative compactness of $f'(U_c)$ follows from \textit{any (even power)
decay} of $U_c$ at infinity, as well as \textit{any (even power) decay} of
the kernel of $(\cm+c)^{-1}$, which is given by
$$
K(x)=\int_{-\infty}^{+\infty}  \f{1}{c+\al(\xi)}\rme^{2\pi\rmi \xi x}\,\rmd\xi.
$$
This last condition is satisfied, under very mild growth requirements of
$\al$.  To see this, an easy integration by
parts argument implies
$$
|K(x)|\leq \f{C}{|x|}\int_{-\infty}^\infty \f{|\al'(\xi)|}{(c+\al(\xi))^2} d\xi
$$
This allows us to formulate  more specific and easily verifiable conditions
which imply that the essential spectrum assumption (b) is satisfied:

\begin{proposition}
\label{prop:47} Let $c>0$ and the wave $U_c$ and the multiplier $\al$ satisfy
\begin{itemize}
\item $U_c$ has some power decay at $\infty$,
\item the Fourier multiplier $\al(\xi)$
\begin{itemize}
\item is a continuous differentiable a.e. function with $\al(\xi)\geq 0$
    and $\lim_{|\xi| \to \infty} \al(\xi)=\infty$,
\item satisfies the estimate
\begin{equation}
\label{o:17}
\int_{-\infty}^{+\infty}  \f{|\al'(\xi)|}{(c+\al(\xi))^2}\,\rmd\xi<\infty.
\end{equation}
\end{itemize}
\end{itemize}
Then $\cm+c>\de^2 \calI$, and $f'(U_c)$ is a relatively compact perturbation
of $\cm+c$. In other words, requirement (b) for the operator
$\cl=\cm+c-f'(U_c)$ is satisfied.
\end{proposition}

It is instructive to consider the spectral stability of the traveling
wave $U_c(x)=4c/(1+c^2 x^2)$ of the Benjamin-Ono equation
$$
\partial_tu-\p_x\left(|\p_x| u-u^2\right)=0.
$$
The orbital stability of these
waves has already been established by \citet{albert:tpa91} and
\citet{albert:scf87}; hence, this wave is spectrally stable. An alternative
approach is to use the approach of \autoref{cor:44} and  use the spectral
information provided by the work of \citet{amick:uar91}), see the remarks after Theorem \ref{theo:frkdv}.

In fact, we have more general result, which is applicable to the so-called
fractional KdV equations (or, fractional Benjamin-Ono equation as they are
referred to in \citep{frank:uan12})
\begin{equation}
\label{frKDV}
\partial_tu -\p_x\left(|\p_x|^{s} u-u^{p+1}\right)=0.
\end{equation}
Using the traveling wave ansatz $U_c(\xi),\,\xi=x-ct$ with $c>0$ in
\eref{frKDV} (and requiring that $U_c$ vanishes at $\pm \infty$) leads us to
the existence equation
\begin{equation}
\label{frKDV:1}
|\p_\xi|^s U_c+ c U_c-U_c^{p+1}=0
\end{equation}
An elementary scaling analysis then shows that $U_c(\xi)=c^{1/p}
Q(c^{1/s}\xi)$, where $Q$ satisfies
\begin{equation}
\label{frKDV:2}
|\p_\xi|^s Q+ Q-Q^{p+1}=0.
\end{equation}
Assume that $0<s<2$, and set
\[
p_{\max}(s)=\begin{cases}
2s/(1-s),\quad & 0<s<1 \\
+\infty,\quad & 1\leq s<2.
\end{cases}
\]
It was recently shown by \citet{frank:uan12} that if $0<p<p_{\max}(s)$, then
\eref{frKDV:2} has an unique (up to translation) ground state solution $Q$
which is positive and bell-shaped, i.e., even and decreasing in $(0,\infty)$.
In addition, the solution $Q$ has  decay as $\xi\to+\infty$; in fact,
$|Q(\xi)|\leq C |\xi|^{-1}$ (see \citep[Lemma~5.6]{frank:uan12}).
Furthermore, its linearized operator,
$$
\calL_+=|\p_\xi|^s+1 - (p+1) Q^p,
$$
satisfies $\rmn(\calL_+)=1$, the kernel is simple with
$\ker(\calL_+)=\Span\{\partial_\xi Q\}$, and the rest of the spectrum is
positive, with a spectral gap at the zero. A simple rescaling then implies
that similar statements hold for the linearized operator
$$
\calL_c=|\p_\xi|^s+c - (p+1) U_c^p.
$$
Going back to the linearized stability of $U_c$, we see that we need to
consider the eigenvalue problem
$$
\p_\xi\calL_c z=\lambda z.
$$
By the decay  of $Q$, the relative compactness of $U_c^p$ follows easily
(indeed, note that  the condition is satisfied for when the multiplier is
$\al(\xi)=|\xi|^s$ for any $s>0$). Thus, we can apply \autoref{cor:44} to
conclude that
\[
K_{\Ham}=1-\rmp(\p_c\langle
U_c,U_c\rangle).
\]
The stability of the traveling wave $U_c$ is equivalent to the positivity of
$\p_c\langle U_c,U_c\rangle$. Since
$$
\p_c\langle U_c,U_c\rangle= const.\left(\f{2}{p}-\f{1}{s}\right) c^{\f{2}{p}-\f{1}{s}-1}
$$
for some positive constant, we have shown the following:

\begin{theorem}
\label{theo:frkdv} Let $0<s<2$ and $0<p<p_{\max}$. The unique ground state
traveling wave solutions $U_c$ are spectrally stable if
$$
0<p<2 s.
$$
Otherwise, if $2s<p<p_{\max}(s)$ there is precisely one positive real
eigenvalue, and the rest of the spectrum is purely imaginary.
\end{theorem}

\begin{remark}
The Benjamin-Ono case corresponds to the case $s=1, p=2$, which is the
borderline case in the above computation. Since in this case $\p_c\langle
U_c,U_c\rangle=0$, the generalized kernel for $\partial_\xi\calL_c$ will have
dimension of at least three. The calculation of the Hamiltonian-Krein index
would have to be modified in order to take into account the fact that the
dimension of the generalized kernel is larger than two. This is a relatively
straightforward exercise which we will leave for the interested reader (e.g.,
see \citep[Index Theorem~2.1]{kapitula:sif12} and the references therein for
an indication as to what modifications must be made). The end result is the
expected one: the index satisfies $K_{\Ham}=0$.
\end{remark}

\begin{remark}
The result of \autoref{theo:frkdv} recovers the classical KdV results for the
limiting case $s=2$, which does not require the theory of \citep{frank:uan12}
in order to conclude the existence of the wave $U_c$. Here we then get the
well-known result of spectral (which in this case implies orbital) stability
for ground state solutions to the generalized KdV whenever $p<4$, and
spectral instability for $p>4$.
\end{remark}

\section{Applications to BBM-like problems}
 
Consider the BBM-type problem
 \begin{equation}
 \label{lin:BBM}
 (\calI+\calM)\partial_tu+\partial_x\left(u+f(u)\right)=0,
 \end{equation}
under the same conditions on the nonlinearity $f$ and the dispersion operator
$\cm$ as in \autoref{cor:44}. These problems have been considered in this
generality by numerous authors; for example, see the paper \cite{lin:son08}
for an extensive discussion. In addition to the requirements on $\calM$
presented in the discussion leading to \autoref{cor:44}, we require that
$\calI+\cm\geq \de^2\calI>0$, which hence possesses a square root and a
bounded inverse\footnote{In the classical BBM model $\cm=-\p_x^2$, in which
case the invertibility of $(1-\p_x^2)^{-1}$ is a well-settled issue}. This
last requirement, in view of the form required of $\cm$ amounts to the
multiplier satisfying the inequality $1+\al(\xi)\geq \de^2$.

After integration in the traveling variable $\xi=x-c t$ the traveling wave
solution $U_c$\footnote{which is henceforth assumed to be homoclinic to zero
at $\xi=\pm \infty$ and in addition, it has at least a power decay} will
satisfy the elliptic PDE
\begin{equation}
\label{s10}
c\cm U_c (\xi)+(c-1) U_c(\xi)-f(U_c)=0.
\end{equation}
Introduce the corresponding linearized operator
 $$
 \cl_0\coloneqq c \cm +(c-1)-f'(U_c).
 $$
Clearly, a differentiation\footnote{Here is where matters that $\cm$ is given
by multiplier, so that $\p_\xi \cm=\cm \p_\xi$} in $\xi$ of the defining
equation  \eref{s10} yields that $\cl_0(\p_\xi U_c)=0$: we henceforth assume
that $\ker(\calL_0)=\Span\{\partial_\xi U_c\}$. Writing $u=U_c(\xi)+v(\xi,t)$
in \eref{lin:BBM} and linearizing yields
\[
 (\calI+\cm)\partial_tv= \p_\xi[c \cm+(c-1)-f'(U_c)]v=\p_\xi \cl_0 v,
\]
which in turn yields the eigenvalue problem
\begin{equation}\label{e:47}
\p_\xi \cl_0 v=\lambda(\calI+\cm)v.
\end{equation}
In order to put this linearized problem in the desired framework, we
introduce the variable
\[
z\coloneqq(\calI+\calM)^{1/2}v.
\]
Taking $(\calI+\cm)^{-1/2}$ on both sides of \eref{e:47} yields the new
eigenvalue problem
\begin{equation}\label{s15}
\p_\xi[(\calI+\cm)^{-1/2} \cl_0 (\calI+\cm)^{-1/2}]z=\lambda z.
\end{equation}

If $\lambda$ is an eigenvalue for \eref{s15} with corresponding eigenfunction
$z$, then the fact that $(\calI+\calM)^{-1/2}$ is a bounded operator implies
that $\lambda$ is an eigenvalue for \eref{e:47} with corresponding
eigenfunction $v=(\calI+\calM)^{-1/2}z$. On the other hand, suppose that
$\lambda$ is a nonzero eigenvalue for \eref{e:47} with corresponding
eigenfunction $v$. Since $(\calI+\calM)v$ necessarily makes sense, it is true
that $z=(\calI+\calM)^{1/2}v$ is well-defined; thus, $\lambda$ is also an
eigenvalue for \eref{s15}. In conclusion, when looking for nonzero
eigenvalues the two spectral problems are equivalent.
Denote
$$
\cl\coloneqq(\calI+\cm)^{-1/2} \cl_0 (\calI+\cm)^{-1/2},
$$
so that the eigenvalue problem associated with \eref{s15} is of the desired
form \eref{5}. We now discuss abstract conditions that would imply
\begin{enumerate}
\item $\sigma_{\ess}(\calL)\subset [\ka_0^2, \infty),\quad \ka_0>0$
\item $\rmn(\calL)=\rmn(\calL_0)$.
\end{enumerate}
The condition (b) is satisfied because  $(\calI+\calM)^{-1/2}$ is a bounded
and symmetric operator. Regarding (a), write
$$
\calL=(\calI+\cm)^{-1/2}[c \cm +(c-1)] (\calI+\cm)^{-1/2} - (\calI+\cm)^{-1/2} f'(U_c) (\calI+\cm)^{-1/2}.
$$
If $(\calI+\cm)^{-1/2} f'(U_c) (\calI+\cm)^{-1/2}$ is a compact operator,
then
$$
\si_{\ess}(\calL)=\si_{\ess}((\calI+\cm)^{-1/2}[c \cm +(c-1)] (\calI+\cm)^{-1/2})=
\mathrm{Range}[\xi \mapsto \f{c\al(\xi)+c-1}{1+\al(\xi)}]\subset [\ka_0^2, \infty).
$$
Indeed, the last inclusion holds because of the continuity of the symbol
$\f{c\al(\xi)+c-1}{1+\al(\xi)}\geq 0$, the fact that $c\al(\xi)+c-1\geq
\de^2$, and the limiting behavior $\lim_{\xi\to \infty}
\f{c\al(\xi)+c-1}{1+\al(\xi)}=c>0$. On the other hand, the compactness of the
operator $(\calI+\cm)^{-1/2} f'(U_c) (\calI+\cm)^{-1/2}$  holds under minimal
assumptions on the decay of $U_c$ and the kernel of $(\calI+\cm)^{-1}$, see Proposition \ref{prop:47}. More precisely, to show that the kernel $K$ of $(\calI+\cm)^{-1/2}$ satisfies $|K(x)|\leq C |x|^{-1}$, we need that the multiplier $\al(\xi)$ satisfies
$$
\int_{-\infty}^\infty \f{|\al'(\xi)|}{(1+\al(\xi))^{3/2}} d\xi<\infty.
$$
This  is certainly satisfied in the cases of interest, $\al(\xi)=|\xi|^s, s>0$.

Continuing with the analysis of the operator $\cl$, we verify by a direct inspection
\[
\cl[ (\calI+\cm)^{1/2} \p_\xi U_c]=(\calI+\cm)^{-1/2} (\cl_0[\p_\xi U_c])=0,
\]
which gives a definitive relationship between the kernels of the two
operators.   In particular, they
have the same dimension. Regarding the Hamiltonian-Krein index for
\eref{s15}, the fact that \\ $(\calI+\cm)^{1/2} \p_\xi
U_c=\p_\xi(\calI+\cm)^{1/2}U_c$ allows us to write
\[
\langle\cl^{-1}\partial_\xi^{-1}[(\calI+\cm)^{1/2} \p_\xi U_c],\partial_\xi^{-1}(\calI+\cm)^{1/2} \p_\xi U_c\rangle=
\langle\calL^{-1}[(\calI+\cm)^{1/2}U_c],(\calI+\cm)^{1/2}U_c\rangle,
\]
which in terms of the original operator yields
\[
\langle\calL^{-1}[(\calI+\cm)^{1/2}U_c],(\calI+\cm)^{1/2}U_c\rangle=
\langle\calL_0^{-1}[(\calI+\cm)U_c],(\calI+\cm)U_c\rangle.
\]
Now, taking a derivative in the variable $c$ in the existence equation
\eref{s10} yields
\[
\cl_0[\p_c U_c]=-(\calI+\cm)U_c\quad\Rightarrow\quad
\cl_0^{-1} [(I+\cm)U_c]=-\p_c U_c,\footnote{This relationship is a verification, via
the Fredholm alternative, that $(I+\cm)U_c\in\ker(\calL_0)^\perp$}
\]
so that
\[
\langle\calL_0^{-1}[(\calI+\cm)U_c],(\calI+\cm)U_c\rangle=
-\frac12\partial_c\langle(\calI+\cm)U_c,U_c\rangle.
\]
Upon applying \autoref{t:index} we have now shown the following:

\begin{theorem}\label{cor:BBM}
Consider the BBM-type equation \eref{lin:BBM}, assume further that $\cl_0$
satisfies
\begin{itemize}
\item $\al(\xi)$ is continuous, $\lim_{|\xi|\to \infty} \al(\xi)=\infty$,
    and $c \al(\xi)+c-1\geq \de^2 >0$ for some $\de>0$
\item $(\calI+ \cm)^{-1/2} f'(U_c)(\calI+ \cm)^{-1/2} $ is a compact operator
    on $L^2(\R)$
\item $\ker(\calL_0)=\Span\{\partial_\xi U_c\}$.
\end{itemize}
The Hamiltonian-Krein index for the spectral problem \eref{e:47} is
\[
K_{\Ham}=\rmn(\calL_0)-
\rmp\left(\partial_c\langle(\calI+\cm) U_c, U_c \rangle\right).
\]
In particular,
\begin{itemize}
\item the wave $U_c$ is spectrally unstable if if $K_{\Ham}$ is odd, i.e.,
    \begin{itemize}
    \item $\rmn(\cl)$ is even and $\partial_c\langle(\calI+\cm) U_c, U_c
        \rangle>0$
    \item $\rmn(\cl)$ is odd and $\partial_c\langle(\calI+\cm) U_c, U_c
        \rangle<0$
    \end{itemize}
\item the wave $U_c$ is spectrally stable if $\rmn(\cl)=1$ and
    $\partial_c\langle(\calI+\cm) U_c, U_c \rangle>0$.
\end{itemize}
\end{theorem}

\begin{remark}
Comparing these results with the corresponding results of
\citet[Theorem~1]{lin:son08}, we see that regarding the requirements on the
multiplier $\al(\xi)$, they are slightly more general than the ones proposed
by Lin (see also \autoref{prop:47}).
\end{remark}

\subsection{Fractional BBM Equations}\label{fracBBM}

We are now ready to characterize the spectral stability of the ground state
traveling wave solutions of the fractional BBM equation, which is given by
\begin{equation}
\label{BBM:1}
\partial_tu+\partial_xu+\partial_t(|\p_x|^s u)+\partial_x(u^{p+1})=0
\end{equation}
(compare with \eref{lin:BBM}). In the traveling wave ansatz $U_c(\xi)$ with
$\xi=x-ct$ we obtain the existence equation
\begin{equation}
\label{BBM:5}
c |\p_\xi|^s U_c + (c-1) U_c - U_c^{p+1}=0.
\end{equation}
Under the assumption $c>1$, which is necessary for the existence of such
waves, we see that matters once again reduce to the solution $Q$ of
\eref{frKDV:2}. Indeed, the solution $U_c$ of \eref{BBM:5} can be written as
$$
U_c(x)=(c-1)^{1/p} Q\left(\left[\frac{c-1}{c}\right]^{1/s}x\right).
$$

Next, we verify that the operator $\cl_0=c|\p_\xi|^s+(c-1)-(p+1)U_c^p$
satisfies the requirements of \autoref{cor:BBM} for any $c>1$ and $s>0$.
Indeed, letting the multiplier satisfy $\al(\xi)=|\xi|^s$, we see that
$c\al(\xi)+c-1\geq c-1=:\delta$. Similar to the claims  of \autoref{prop:47},
we use the $|\xi|^{-1}$ decay of $Q$ (and $U_c$, respectively) to conclude
that the operator $(I+|\p_\xi|^s)^{-1/2} U_c^p (I+|\p_\xi|^s)^{-1/2}$ is a
compact operator on $L^2(\R)$. Finally, $\ker(\cl_0)=\Span\{\p_\xi U_c\}$ and
$\rmn(\cl_0)=1$ is a consequence\footnote{by a simple change of variables} of
the corresponding statement for the operator $\calL_+$ of
\citet{frank:uan12}.

Thus, we may apply our formula for the Hamilton-Krein index from
\autoref{cor:BBM}. To this end, we need to compute $ \p_c
\dpr{(\calI+\cm)U_c}{U_c}, $ which from the defining equation \eref{BBM:5} is
expressible as follows
\[
\dpr{(\calI+\cm)U_c}{U_c}=\f{1}{c}\dpr{U_c+U_c^{p+1}}{U_c}\\
=(c-1)^{\f{2}{p}-\f{1}{s}} c^{\f{1}{s}-1}
\langle Q,Q\rangle  +
(c-1)^{1+\f{2}{p}-\f{1}{s}}c^{\f{1}{s}-1}\langle Q^{p+1},Q\rangle.
\]
It follows that
$$
\p_c \dpr{(\calI+\cm)U_c}{U_c}= (c-1)^{ \f{2}{p}-\f{1}{s}-1}
c^{\f{1}{s}-2}[\left(c(2/p-1)+1-1/s\right)\langle Q,Q\rangle+
(2c/p+1-1/s)\langle Q^{p+1},Q\rangle].
$$
Since $c>1$, upon simplifying we have that
\[
\p_c \dpr{(\calI+\cm)U_c}{U_c}\propto
\left[(2-p)sc+(s-1)p\right]\langle Q,Q\rangle+
\left[2sc+(s-1)p\right]\langle Q^{p+1},Q\rangle.
\]
Now, using the existence equation \eref{frKDV:2} we have that
\[
\langle Q^{p+1},Q\rangle=\langle Q,Q\rangle+\langle|\partial_\xi|^sQ,Q\rangle=
\langle Q,Q\rangle+\langle|\partial_\xi|^{s/2}Q,|\partial_\xi|^{s/2}Q\rangle;
\]
thus, we can rewrite the above to say that
\[
\p_c \dpr{(\calI+\cm)U_c}{U_c}\propto
\left[(4-p)sc+2(s-1)p\right]\langle Q,Q\rangle+
\left[2sc+(s-1)p\right]\langle|\partial_\xi|^{s/2}Q,|\partial_\xi|^{s/2}Q\rangle.
\]
This allows us to conclude with the following:

\begin{theorem}\label{frac:BBM}
Let $0<s<2$, $p\in (0, p_{\max})$ and $c>1$. The unique (up to translation)
ground state $U_c$ of the fractional BBM equation (which is a solution of
\eref{BBM:5}) is spectrally stable if and only if
\[
\left[(4-p)sc+2(s-1)p\right]\langle Q,Q\rangle>
-\left[2sc+(s-1)p\right]\langle|\partial_\xi|^{s/2}Q,|\partial_\xi|^{s/2}Q\rangle.
\]
In particular, the wave is spectrally stable if $1\le s<2$ and $0<p\le4$. If
the inequality is reversed the linearized eigenvalue problem has precisely
one positive real eigenvalue, and the rest of the spectrum is purely
imaginary.
\end{theorem}

\phantomsection                                         
\addcontentsline{toc}{section}{\refname}                


\begin{thebibliography}{18}
\providecommand{\natexlab}[1]{#1} \providecommand{\url}[1]{\texttt{#1}}
\expandafter\ifx\csname urlstyle\endcsname\relax
  \providecommand{\doi}[1]{doi: #1}\else
  \providecommand{\doi}{doi: \begingroup \urlstyle{rm}\Url}\fi

\bibitem[Albert and Bona(1991)]{albert:tpa91} J.~Albert and J.~Bona.
\newblock Total positivity and the stability of internal waves in stratified
  fluids of finite depth.
\newblock \emph{IMA J. Appl. Math.}, 46\penalty0 (12):\penalty0 119, 1991.

\bibitem[Albert et~al.(1987)Albert, Bona, and Henry]{albert:scf87} J.~Albert,
    J.~Bona, and D.~Henry.
\newblock Sufficient conditions for stability of solitary-wave solutions of
  model equations for long waves.
\newblock \emph{Phys. D}, 24\penalty0 (13):\penalty0 343366, 1987.

\bibitem[Amick and Toland(1991)]{amick:uar91} C.~Amick and J.~Toland.
\newblock Uniqueness and related analytic properties for the {Benjamin-Ono}
  equation - a nonlinear {N}eumann problem in the plane.
\newblock \emph{Acta Math.}, 167\penalty0 (1-2):\penalty0 107126, 1991.

\bibitem[Bona et~al.(1987)Bona, Souganidis, and Strauss]{bona:sai87} J.~Bona,
    P.~Souganidis, and W.~Strauss.
\newblock Stability and instability of solitary waves of {K}orteweg-de {V}ries
  type.
\newblock \emph{Proc. R. Soc. London A}, 411:\penalty0 395--412, 1987.

\bibitem[Bronski et~al.(2011)Bronski, Johnson, and Kapitula]{bronski:ait11}
    J.~Bronski, M.~Johnson, and T.~Kapitula.
\newblock An index theorem for the stability of periodic traveling waves of
  {KdV} type.
\newblock \emph{Proc. Roy. Soc. Edinburgh: Section A}, 141\penalty0
  (6):\penalty0 1141--1173, 2011.

\bibitem[Deconinck and Kapitula()]{deconinck:ots10} B.~Deconinck and
    T.~Kapitula.
\newblock On the spectral and orbital stability of spatially periodic
  stationary solutions of generalized {K}orteweg-de {V}ries equations.
\newblock submitted.

\bibitem[Frank and Lenzmann(2012)]{frank:uan12} R.~Frank and E.~Lenzmann.
\newblock Uniqueness and non-degeneracy of ground states for
  {$(-\Delta)^sQ+Q-Q^{\alpha+1}=0$ in $\R$}.
\newblock to appear in \textit{Acta Math.}, 2012.

\bibitem[H\v{a}r\v{a}gu\c{s} and Kapitula(2008)]{haragus:ots08}
    M.~H\v{a}r\v{a}gu\c{s} and T.~Kapitula.
\newblock On the spectra of periodic waves for infinite-dimensional
  {H}amiltonian systems.
\newblock \emph{Physica D}, 237\penalty0 (20):\penalty0 2649--2671, 2008.

\bibitem[Kapitula and Promislow(2012{\natexlab{a}})]{kapitula:ait12}
    T.~Kapitula and K.~Promislow.
\newblock \emph{An Introduction to Spectral and Dynamical Stability: Theory and
  Applications}.
\newblock Springer-Verlag, 2012{\natexlab{a}}.
\newblock in preparation.

\bibitem[Kapitula and Promislow(2012{\natexlab{b}})]{kapitula:sif12}
    T.~Kapitula and K.~Promislow.
\newblock Stability indices for constrained self-adjoint operators.
\newblock \emph{Proc. Amer. Math. Soc.}, 140\penalty0 (3):\penalty0 865--880,
  2012{\natexlab{b}}.

\bibitem[Kapitula et~al.(2004)Kapitula, Kevrekidis, and
  Sandstede]{kapitula:cev04}
T.~Kapitula, P.~Kevrekidis, and B.~Sandstede.
\newblock Counting eigenvalues via the {K}rein signature in
  infinite-dimensional {H}amiltonian systems.
\newblock \emph{Physica D}, 195\penalty0 (3\&4):\penalty0 263--282, 2004.

\bibitem[Kapitula et~al.(2005)Kapitula, Kevrekidis, and
  Sandstede]{kapitula:ace05}
T.~Kapitula, P.~Kevrekidis, and B.~Sandstede.
\newblock Addendum: {C}ounting eigenvalues via the {K}rein signature in
  infinite-dimensional {H}amiltonian systems.
\newblock \emph{Physica D}, 201\penalty0 (1\&2):\penalty0 199--201, 2005.

\bibitem[Lax(2002)]{lax:fa02} P.~Lax.
\newblock \emph{Functional Analysis}.
\newblock Wiley-Interscience, 2002.

\bibitem[Lin(2008)]{lin:son08} Z.~Lin.
\newblock Stability of nonlinear dispersive solitary waves.
\newblock \emph{J. Functional Anal.}, 255\penalty0 (5):\penalty0 1191--1224,
  2008.

\bibitem[Pego and Weinstein(1992)]{pego:eis92} R.~Pego and M.~Weinstein.
\newblock Eigenvalues, and instabilities of solitary waves.
\newblock \emph{Phil. Trans. R. Soc. Lond. A}, 340:\penalty0 47--94, 1992.

\bibitem[Pego and Weinstein(1993)]{pego:efm93} R.~Pego and M.~Weinstein.
\newblock Evans' function, and {M}elnikov's integral, and solitary wave
  instabilities.
\newblock In \emph{Differential Equations with Applications to Mathematical
  Physics}, pages 273--286. Academic Press, Boston, 1993.

\bibitem[Reed and Simon(1978)]{reed:aoo78} M.~Reed and B.~Simon.
\newblock \emph{Methods of Modern Mathematical Physics IV: Analysis of
  Operators}.
\newblock Academic Press, Inc., 1978.

\bibitem[Reed and Simon(1980)]{reed:fa80} M.~Reed and B.~Simon.
\newblock \emph{Methods of Modern Mathematical Physics I: Functional Analysis}.
\newblock Academic Press, Inc., 1980.

\end{thebibliography}
\bibliographystyle{plainnat}

\end{document}